\documentclass[reqno]{amsart}
\usepackage{amsmath}
\usepackage{amsfonts}

\newtheorem{theorem}{Theorem}
\theoremstyle{plain}

\newtheorem{corollary}{Corollary}

\newtheorem{lemma}{Lemma}

\newtheorem{remark}{Remark}

\numberwithin{equation}{section}
\input{tcilatex}

\begin{document}
\title[]{Some New Generalized Results on Ostrowski Type Integral
Inequalities With Application}
\author{$^{1}$A. Qayyum}
\address{$^{1}$Department of fundamental and applied sciences, Universiti
Teknologi Petronas, Tronoh, Malaysia. }
\email{atherqayyum@gmail.com}
\author{$^{2}$M. Shoaib}
\address{$^{2}$Department of Mathematics, University of Hail, 2440, Saudi
Arabia}
\email{safridi@gmail.com}
\author{ $^{1}$Ibrahima Faye}
\email{ibrahima\_faye@petronas.com.my}
\date{}
\thanks{This paper is in final form and no version of it will be submitted
for publication elsewhere.}

\begin{abstract}
The aim of this paper is to establish some new inequalities similar to the
Ostrowski's inequalities which are more generalized than the inequalities of
Dragomir and Cerone. The current article obtains bounds for the deviation of
a function from a combination of integral means over the end intervals
covering the entire interval. Some new purterbed results are obtained.
Application for cumulative distribution function is also discussed.
\end{abstract}

\maketitle

\section{Introduction}

In 1938, Ostrowski \cite{15} established an interesting integral inequality
associated with differentiable mappings. This Ostrowski inequality has
powerful applications in numerical integration, probability and optimization
theory, stochastic, statistics, information and integral operator theory. A
number of Ostrowski type inequalities have been derived by Cerone \cite{1}, 
\cite{2} and Cheng \cite{3} with applications in Numerical analysis and
Probability. Dragomir et.al \cite{5} combined Ostrowski and Gr\"{u}ss
inequality to give a new inequality which they named Ostrowski-Gr\"{u}ss
type inequality. Milovanovi\'{c} and Pecari\'{c} \cite{13} gave the first
generalization of Ostrowski's inequality. More recent results concerning the
generalizations of Ostrowski inequality are given by Liu \cite{12}, Hussain 
\cite{10} and Qayyum \cite{17}. In this paper, we will extend and generalize
the results of Cerone \cite{1} and Dragomir et.al \cite{5}-\cite{8} by using
a new kernel.

Let $S\left( f;a,b\right) $ be defined by%
\begin{equation}
S\left( f;a,b\right) :=f\left( x\right) -M\left( f;a,b\right) ,  \label{2-A}
\end{equation}%
where%
\begin{equation}
M\left( f;a,b\right) =\frac{1}{b-a}\int\limits_{a}^{b}f\left( x\right) dx
\label{3}
\end{equation}%
is the integral mean of $f$ over $[a,b]$. The functional $S\left(
f;a,b\right) $ represents the deviation of $f\left( x\right) $ from its
integral mean over $\left[ a,b\right] .$

Ostrowski \cite{15} proved the following integral inequality:

\begin{theorem}
Let \ \ $f\ $: $\left[ a,b\right] \rightarrow 
\mathbb{R}
$ \ be\ continuous on $\left[ a,b\right] $ and differentiable on $\left(
a,b\right) ,$ whose derivative $f^{\prime }:\left( a,b\right) \rightarrow 
\mathbb{R}
$ is bounded on $\ \left( a,b\right) ,$ i.e. $\left\Vert f^{\prime
}\right\Vert _{\infty }=\sup_{t\in \left[ a,b\right] }\left\vert f^{\prime
}\left( t\right) \right\vert <\infty ,$ then%
\begin{equation}
\left\vert S\left( f;a,b\right) \right\vert \leq \left[ \left( \frac{b-a}{2}%
\right) ^{2}+\left( x-\frac{a+b}{2}\right) ^{2}\right] \frac{M}{b-a}
\label{1}
\end{equation}%
for all $x\in \left[ a,b\right] $.
\end{theorem}

In a series of papers, Dragomir et al \cite{5}-\cite{8} proved (\ref{1}) and
some of it's variants for $f^{\text{ }\prime }\in L_{p}\left[ a,b\right] $
when $p\geq 1,$ for Lebesgue norms making use of a peano kernel.

If we assume that $f^{\text{ }\prime }\in L_{\infty }\left[ a,b\right] $ and 
$\left\Vert f^{\prime }\right\Vert _{\infty }=\underset{t\in \left[ a,b%
\right] }{ess}\left\vert f^{\prime }\left( t\right) \right\vert $ then $M$
in (\ref{1}) may be replaced by $\left\Vert f^{\prime }\right\Vert _{\infty
}.$

Dragomir et al \cite{5}-\cite{8} utilizing an integration by parts argument$%
, $ obtained%
\begin{eqnarray}
&&\left\vert S\left( f;a,b\right) \right\vert  \label{4} \\
&&  \notag \\
&\leq &\left\{ 
\begin{array}{l}
\frac{1}{b-a}\left[ \left( \frac{b-a}{2}\right) ^{2}+\left( x-\frac{a+b}{2}%
\right) ^{2}\right] \left\Vert f^{\prime }\right\Vert _{\infty },f^{\text{ }%
\prime }\in L_{\infty }\left[ a,b\right] , \\ 
\\ 
\frac{1}{b-a}\left[ \frac{\left( x-a\right) ^{q+1}+\left( b-x\right) ^{q+1}}{%
q+1}\right] ^{\frac{1}{q}}\left\Vert f^{\prime }\right\Vert _{p},f^{\text{ }%
\prime }\in L_{p}\left[ a,b\right] ,p>1,\frac{1}{p}+\frac{1}{q}=1, \\ 
\\ 
\frac{1}{b-a}\left[ \frac{b-a}{2}+\left\vert x-\frac{a+b}{2}\right\vert %
\right] \left\Vert f^{\prime }\right\Vert _{1},f^{\prime }\in L_{1}\left[ a,b%
\right] ,%
\end{array}%
\right.  \notag
\end{eqnarray}%
where $f\ $: $\left[ a,b\right] \rightarrow 
\mathbb{R}
$ is absolutely continuous\ on $\left[ a,b\right] $ and the constants $\frac{%
1}{4},\left[ \frac{1}{q+1}\right] ^{\frac{1}{q}}$ and $\frac{1}{2}$ are
sharp. In \cite{16}, Pachpatte established \v{C}eby\^{s}ev type inequalities
by using Pecari\'{c}'s extension of the Montgomery identity \cite{19}.
Cerone \cite{1}, proved the following inequality:

\begin{lemma}
Let \ \ $f\ $: $\left[ a,b\right] \rightarrow 
\mathbb{R}
$ \ be\ absolutely continuous function. Define%
\begin{equation}
\tau \left( x;\alpha ,\beta \right) :=f\left( x\right) -\frac{1}{\alpha
+\beta }\left[ \alpha M\left( f;a,x\right) +\beta M\left( f;x,b\right) %
\right] ,  \label{5}
\end{equation}%
where%
\begin{eqnarray}
&&\left\vert \tau \left( x;\alpha ,\beta \right) \right\vert   \label{6} \\
&&  \notag \\
&\leq &\left\{ 
\begin{array}{l}
\frac{1}{2\left( \alpha +\beta \right) }\left[ \alpha \left( x-a\right)
+\beta \left( b-x\right) \right] \left\Vert f^{\prime }\right\Vert _{\infty
},f^{\text{ }\prime }\in L_{\infty }\left[ a,b\right] , \\ 
\\ 
\frac{1}{\left( \alpha +\beta \right) \left( q+1\right) ^{\frac{1}{q}}}\left[
\alpha ^{q}\left( x-a\right) +\beta ^{q}\left( b-x\right) \right] ^{\frac{1}{%
q}}\left\Vert f^{\prime }\right\Vert _{p} \\ 
,f^{\text{ }\prime }\in L_{p}\left[ a,b\right] ,p>1,\frac{1}{p}+\frac{1}{q}%
=1, \\ 
\\ 
\frac{1}{2}\left( 1+\frac{\left\vert \alpha -\beta \right\vert }{\alpha
+\beta }\right) \left\Vert f^{\prime }\right\Vert _{1},f^{\prime }\in L_{1}%
\left[ a,b\right] ,%
\end{array}%
\right.   \notag
\end{eqnarray}%
where the usual $L_{p}$ norms $\left\Vert k\right\Vert _{p}$defined for a
function $k\in L_{p}\left[ a,b\right] $ as follows:%
\begin{equation*}
\left\Vert k\right\Vert _{\infty }:=ess\sup_{t\in \left[ a,b\right]
}\left\vert k\left( t\right) \right\vert 
\end{equation*}%
and%
\begin{equation*}
\left\Vert k\right\Vert _{p}:=\left( \int\limits_{a}^{b}\left\vert k\left(
t\right) \right\vert ^{p}dt\right) ^{\frac{1}{p}},1\leq p<\infty .
\end{equation*}
\end{lemma}

With the help of two different kernels (\ref{7}) and (\ref{8}) given below,
we extended the version of Cerone \cite{1} and Dragomir's result \cite{5}-%
\cite{8}.

\begin{lemma}
Let $P\left( x,.\right) \ $: $\left[ a,b\right] \rightarrow 
\mathbb{R}
,$ the peano type kernel is given by%
\begin{equation}
p(x,t)=\left\{ 
\begin{array}{c}
\frac{\alpha }{\alpha +\beta }\frac{1}{x-a}\left[ t-\left( a+h\frac{b-a}{2}%
\right) \right] ,\text{ }a\leq t\leq x, \\ 
\\ 
\frac{\beta }{\alpha +\beta }\frac{1}{b-x}\left[ t-\left( b-h\frac{b-a}{2}%
\right) \right] ,\text{ }x<t\leq b,%
\end{array}%
\right.  \label{7}
\end{equation}%
Then,%
\begin{eqnarray}
\left\vert \tau \left( x;\alpha ,\beta \right) \right\vert &=&\frac{1}{%
\alpha +\beta }\left[ 
\begin{array}{c}
\frac{\alpha }{x-a}\left\{ x-\left( a+h\frac{b-a}{2}\right) \right\} \\ 
\\ 
-\frac{\beta }{b-x}\left\{ x-\left( b-h\frac{b-a}{2}\right) \right\}%
\end{array}%
\right] f\left( x\right)  \label{6-A} \\
&&+\frac{h}{\alpha +\beta }\left( \frac{b-a}{2}\right) \left( \frac{\alpha }{%
x-a}f\left( a\right) +\frac{\beta }{b-x}f\left( b\right) \right)  \notag \\
&&-\frac{1}{\alpha +\beta }\left[ \frac{\alpha }{x-a}\int\limits_{a}^{x}f%
\left( t\right) dt+\frac{\beta }{b-x}\int\limits_{x}^{b}f\left( t\right) dt%
\right]  \notag \\
&\leq &\left\{ 
\begin{array}{l}
\left( 
\begin{array}{c}
\frac{\alpha }{x-a}\left\{ \frac{\left( x-a\right) ^{2}}{4}+\left[ \left( a+h%
\frac{b-a}{2}\right) -\frac{a+x}{2}\right] ^{2}\right\} \\ 
\\ 
+\frac{\beta }{b-x}\left\{ \frac{\left( b-x\right) ^{2}}{4}+\left[ \left( b-h%
\frac{b-a}{2}\right) -\frac{x+b}{2}\right] ^{2}\right\}%
\end{array}%
\right) \frac{1}{\left( \alpha +\beta \right) }\left\Vert f^{\prime
}\right\Vert _{\infty } \\ 
,\text{ \ \ \ }f^{\text{ }\prime }\in L_{\infty }\left[ a,b\right] \\ 
\\ 
\left[ 
\begin{array}{c}
\frac{\alpha ^{q}}{\left( x-a\right) ^{q}}\left\{ \left( x-\left( a+h\frac{%
b-a}{2}\right) \right) ^{q+1}-\left( h\frac{a-b}{2}\right) ^{q+1}\right\} \\ 
\\ 
+\frac{\beta ^{q}}{\left( b-x\right) ^{q}}\left\{ \left( b-\left( x+h\frac{%
b-a}{2}\right) \right) ^{q+1}-\left( h\frac{a-b}{2}\right) ^{q+1}\right\}%
\end{array}%
\right] ^{^{\frac{1}{q}}}\frac{1}{\left( q+1\right) ^{\frac{1}{q}}\left(
\alpha +\beta \right) }\left\Vert f^{\prime }\right\Vert _{p}, \\ 
\\ 
f^{\text{ }\prime }\in L_{p}\left[ a,b\right] ,p>1,\frac{1}{p}+\frac{1}{q}=1,
\\ 
\\ 
\left( 
\begin{array}{c}
\left( \alpha +\beta \right) -h\frac{b-a}{2}\left[ \frac{\alpha \left(
b-x\right) +\beta \left( x-a\right) }{\left( x-a\right) \left( b-x\right) }%
\right] \\ 
\\ 
+\left\vert \left( \alpha -\beta \right) +h\frac{b-a}{2}\left[ \frac{\beta
\left( x-a\right) -\alpha \left( b-x\right) }{\left( x-a\right) \left(
b-x\right) }\right] \right\vert%
\end{array}%
\right) \frac{\left\Vert f^{\prime }\right\Vert _{1}}{2\left( \alpha +\beta
\right) }.%
\end{array}%
\right.  \notag
\end{eqnarray}
\end{lemma}

\begin{lemma}
Denote by $P\left( x,.\right) \ $: $\left[ a,b\right] \rightarrow 
\mathbb{R}
$ the kernel is given by%
\begin{equation}
P(x,t):=\left\{ 
\begin{array}{c}
\frac{\alpha }{2\left( \alpha +\beta \right) \left( x-a\right) }\left(
t-a\right) ^{2},\text{ }a\leq t\leq x, \\ 
\\ 
\frac{\beta }{2\left( \alpha +\beta \right) \left( b-x\right) }\left(
t-b\right) ^{2},\text{ }x<t\leq b,%
\end{array}%
\right.  \label{8}
\end{equation}%
Then,%
\begin{eqnarray}
&&\left\vert \tau \left( x;\alpha ,\beta \right) \right\vert  \label{6-B} \\
&=&\frac{1}{2\left( \alpha +\beta \right) }\left[ \alpha \left( x-a\right)
-\beta \left( b-x\right) \right] f^{\text{ }\prime }\left( x\right) -f\left(
x\right)  \notag \\
&&+\frac{1}{\alpha +\beta }\left[ \alpha M\left( f;a,x\right) +\beta M\left(
f;x,b\right) \right] ,  \notag \\
&&  \notag \\
&\leq &\left\{ 
\begin{array}{l}
\left[ \alpha \left( x-a\right) ^{2}+\beta \left( b-x\right) ^{2}\right] 
\frac{\left\Vert f^{\text{ }\prime \prime }\right\Vert _{\infty }}{6\left(
\alpha +\beta \right) },f^{\text{ }\prime \prime }\in L_{\infty }\left[ a,b%
\right] \\ 
\\ 
\frac{1}{\left( 2q+1\right) ^{\frac{1}{q}}}\left[ \alpha ^{q}\left(
x-a\right) ^{q+1}+\beta ^{q}\left( b-x\right) ^{q+1}\right] ^{^{\frac{1}{q}}}%
\frac{\left\Vert f^{\text{ }\prime \prime }\right\Vert _{p}}{2\left( \alpha
+\beta \right) }, \\ 
\\ 
f^{\text{ }\prime \prime }\in L_{p}\left[ a,b\right] ,p>1,\frac{1}{p}+\frac{1%
}{q}=1, \\ 
\\ 
\left( \alpha \left( x-a\right) +\beta \left( b-x\right) +\left\vert \alpha
\left( x-a\right) -\beta \left( b-x\right) \right\vert \right) \frac{%
\left\Vert f^{\text{ }\prime \prime }\right\Vert _{1}}{4\left( \alpha +\beta
\right) } \\ 
,f^{\text{ }\prime \prime }\in L_{1}\left[ a,b\right] .%
\end{array}%
\right.  \notag
\end{eqnarray}
\end{lemma}

Using a generalized form of (\ref{8}), we constructed a number of new
results for twice differentiable functions. These results are given in Lemma
4 \ and theorem 2 which are more generalized by (\ref{6-A})-(\ref{6-B}).
These generalized inequalities will have applications in approximation
theory, probability theory and numerical analysis. We will show in our paper
an application of the obtained inequalities for cumulative distribution
function.

\section{Main Results}

We will start our main result with this lemma.

\begin{lemma}
Let \ \ $f\ $: $\left[ a,b\right] \rightarrow 
\mathbb{R}
$ \ be\ an absolutely continuous mapping. Denote by $P\left( x,.\right) \ $: 
$\left[ a,b\right] \rightarrow 
\mathbb{R}
$ the kernel $P(x,t)$ is given by%
\begin{equation}
P(x,t)=\left\{ 
\begin{array}{c}
\frac{\alpha }{\alpha +\beta }\frac{1}{x-a}\frac{1}{2}\left[ t-\left( a+h%
\frac{b-a}{2}\right) \right] ^{2},\text{ }a\leq t\leq x, \\ 
\\ 
\frac{\beta }{\alpha +\beta }\frac{1}{b-x}\frac{1}{2}\left[ t-\left( b-h%
\frac{b-a}{2}\right) \right] ^{2},\text{ }x<t\leq b,%
\end{array}%
\right.   \label{9}
\end{equation}%
for all $x\in \left[ a+h\frac{b-a}{2},b-h\frac{b-a}{2}\right] $ and $h\in %
\left[ 0,1\right] ,$where $\alpha ,\beta \in 
\mathbb{R}
$ are non negative and not both zero. Before we state and prove our main
theorem, we will prove the following identity:%
\begin{eqnarray}
\int\limits_{a}^{b}P(x,t)f^{\text{ }\prime \prime }(t)dt &=&\frac{1}{2\left(
\alpha +\beta \right) }\left[ \frac{\alpha }{x-a}\left( x-\left( a+h\frac{b-a%
}{2}\right) \right) ^{2}\right.   \label{10} \\
&&\left. -\frac{\beta }{b-x}\left( x-\left( b-h\frac{b-a}{2}\right) \right)
^{2}\right] f^{\text{ }\prime }\left( x\right)   \notag \\
&&-\frac{1}{\left( \alpha +\beta \right) }\left[ \frac{\alpha }{x-a}\left(
x-\left( a+h\frac{b-a}{2}\right) \right) \right.   \notag \\
&&\left. -\frac{\beta }{b-x}\left( x-\left( b-h\frac{b-a}{2}\right) \right) %
\right] f\left( x\right)   \notag
\end{eqnarray}%
\begin{eqnarray*}
&&-\frac{1}{\alpha +\beta }h\frac{b-a}{2}\left[ \frac{\alpha }{x-a}f\left(
a\right) +\frac{\beta }{b-x}f\left( b\right) \right]  \\
&&+\frac{1}{\alpha +\beta }h^{2}\frac{\left( b-a\right) ^{2}}{8}\left[ \frac{%
\beta }{b-x}f^{\text{ }\prime }\left( b\right) -\frac{\alpha }{x-a}f^{\text{ 
}\prime }\left( a\right) \right]  \\
&&+\frac{1}{\alpha +\beta }\left[ \frac{\alpha }{x-a}\int\limits_{a}^{x}f%
\left( t\right) dt+\frac{\beta }{b-x}\int\limits_{x}^{b}f\left( t\right) dt%
\right] .
\end{eqnarray*}
\end{lemma}

\begin{proof}
From (\ref{9}), we have%
\begin{eqnarray*}
\int\limits_{a}^{b}P(x,t)f^{\text{ }\prime \prime }(t)dt &=&\frac{\alpha }{%
\left( \alpha +\beta \right) }\frac{1}{x-a}\int\limits_{a}^{x}\frac{\left[
t-\left( a+h\frac{b-a}{2}\right) \right] ^{2}}{2}f^{\text{ }\prime \prime
}(t)dt \\
&&+\frac{\beta }{\left( \alpha +\beta \right) }\frac{1}{b-x}%
\int\limits_{x}^{b}\frac{\left[ t-\left( b-h\frac{b-a}{2}\right) \right] ^{2}%
}{2}f^{\text{ }\prime \prime }(t)dt.
\end{eqnarray*}%
After simplification, we get the required identity (\ref{10})$.$
\end{proof}

We now give our main theorem.

\begin{theorem}
Let \ \ $f\ $: $\left[ a,b\right] \rightarrow 
\mathbb{R}
$ \ be\ an absolutely continuous mapping. Define%
\begin{eqnarray}
\tau \left( x;\alpha ,\beta \right)  &:&=\frac{1}{2\left( \alpha +\beta
\right) }\left[ \frac{\alpha }{x-a}\left( x-\left( a+h\frac{b-a}{2}\right)
\right) ^{2}\right.   \label{11} \\
&&\left. -\frac{\beta }{b-x}\left( x-\left( b-h\frac{b-a}{2}\right) \right)
^{2}\right] f^{\text{ }\prime }\left( x\right)   \notag \\
&&-\frac{1}{\left( \alpha +\beta \right) }\left[ \frac{\alpha }{x-a}\left(
x-\left( a+h\frac{b-a}{2}\right) \right) \right.   \notag \\
&&\left. -\frac{\beta }{b-x}\left( x-\left( b-h\frac{b-a}{2}\right) \right) %
\right] f\left( x\right)   \notag \\
&&-\frac{1}{\alpha +\beta }h\frac{b-a}{2}\left[ \frac{\alpha }{x-a}f\left(
a\right) +\frac{\beta }{b-x}f\left( b\right) \right]   \notag \\
&&+\frac{1}{\alpha +\beta }h^{2}\frac{\left( b-a\right) ^{2}}{8}\left[ \frac{%
\beta }{b-x}f^{\text{ }\prime }\left( b\right) -\frac{\alpha }{x-a}f^{\text{ 
}\prime }\left( a\right) \right]   \notag \\
&&+\frac{1}{\alpha +\beta }\left[ \alpha M\left( f;a,x\right) +\beta M\left(
f;x,b\right) \right] ,  \notag
\end{eqnarray}%
where $M\left( f;a,b\right) $ is the integral mean defined in (\ref{3}), then%
\begin{eqnarray}
&&\left\vert \tau \left( x;\alpha ,\beta \right) \right\vert   \label{12} \\
&&  \notag \\
&\leq &\left\{ 
\begin{array}{l}
\left[ 
\begin{array}{c}
\frac{\alpha }{x-a}\left\{ \left[ x-\left( a+h\frac{b-a}{2}\right) \right]
^{3}+\left( h\frac{b-a}{2}\right) ^{3}\right\}  \\ 
-\frac{\beta }{b-x}\left\{ \left( h\frac{b-a}{2}\right) ^{3}+\left[ x-\left(
b-h\frac{b-a}{2}\right) \right] ^{3}\right\} 
\end{array}%
\right] \frac{\left\Vert f^{\text{ }\prime \prime }\right\Vert _{\infty }}{%
6\left( \alpha +\beta \right) },\text{ \ }f^{\text{ }\prime \prime }\in
L_{\infty }\left[ a,b\right] , \\ 
\\ 
\left[ 
\begin{array}{c}
\frac{\alpha ^{q}}{\left( x-a\right) ^{q}}\left\{ \left[ x-\left( a+h\frac{%
b-a}{2}\right) \right] ^{2q+1}-\left( h\frac{a-b}{2}\right) ^{2q+1}\right\} 
\\ 
+\frac{\beta ^{q}}{\left( b-x\right) ^{q}}\left\{ \left( h\frac{a-b}{2}%
\right) ^{2q+1}-\left[ x-\left( b-h\frac{b-a}{2}\right) \right]
^{2q+1}\right\} 
\end{array}%
\right] ^{\frac{1}{q}}\frac{\left\Vert f^{\text{ }\prime \prime }\right\Vert
_{p}}{2\left( 2q+1\right) ^{\frac{1}{q}}\left( \alpha +\beta \right) }, \\ 
\\ 
f^{\text{ }\prime \prime }\in L_{p}\left[ a,b\right] ,p>1,\frac{1}{p}+\frac{1%
}{q}=1, \\ 
\\ 
\left\{ 
\begin{array}{c}
\alpha \left( x-a\right) +\beta \left( b-x\right) -h\left( b-a\right) \left(
\alpha +\beta \right)  \\ 
+\frac{h^{2}\left( b-a\right) }{2}\left[ \frac{\alpha }{x-a}+\frac{\beta }{%
b-x}\right]  \\ 
+\left\vert 
\begin{array}{c}
\beta \left( b-x\right) -\alpha \left( x-a\right) +h\left( b-a\right) \left(
\alpha -\beta \right)  \\ 
+\frac{h^{2}\left( b-a\right) }{2}\left[ \frac{\beta }{b-x}-\frac{\alpha }{%
x-a}\right] 
\end{array}%
\right\vert 
\end{array}%
\right\} \frac{\left\Vert f^{\text{ }\prime \prime }\right\Vert _{1}}{%
4\left( \alpha +\beta \right) },f^{\text{ }\prime \prime }\in L_{1}\left[ a,b%
\right] 
\end{array}%
\right.   \notag
\end{eqnarray}%
for all $x\in \left[ a,b\right] $ $,$where $\left\Vert k\right\Vert $ is the
usual Lebesgue norm for $k\in L\left[ a,b\right] $ with%
\begin{equation*}
\left\Vert k\right\Vert _{\infty }:=ess\sup_{t\in \left[ a,b\right]
}\left\vert k\left( t\right) \right\vert <\infty 
\end{equation*}%
and%
\begin{equation*}
\left\Vert k\right\Vert _{p}:=\left( \int\limits_{a}^{b}\left\vert k\left(
t\right) \right\vert ^{p}dt\right) ^{\frac{1}{p}},1\leq p<\infty 
\end{equation*}
\end{theorem}

\begin{proof}
Taking the modulus of (\ref{10}) $\ $and using (\ref{11}) and (\ref{3}), we
have%
\begin{equation}
\left\vert \tau \left( x;\alpha ,\beta \right) \right\vert =\left\vert
\int\limits_{a}^{b}P(x,t)f^{\text{ }\prime \prime }(t)dt\right\vert \leq
\int\limits_{a}^{b}\left\vert P(x,t)\right\vert \left\vert f^{\text{ }\prime
\prime }(t)\right\vert dt.  \label{13}
\end{equation}

Therefore, for $f^{\text{ }\prime \prime }\in L_{\infty }\left[ a,b\right] $
we obtain%
\begin{equation*}
\left\vert \tau \left( x;\alpha ,\beta \right) \right\vert \leq \left\Vert
f^{\text{ }\prime \prime }\right\Vert _{\infty
}\int\limits_{a}^{b}\left\vert P(x,t)\right\vert dt.
\end{equation*}%
Now let us observe that%
\begin{eqnarray*}
&&\int\limits_{a}^{b}\left\vert P(x,t)\right\vert dt \\
&=&\frac{\alpha }{2\left( \alpha +\beta \right) \left( x-a\right) }%
\int\limits_{a}^{x}\left[ t-\left( a+h\frac{b-a}{2}\right) \right] ^{2}dt \\
&&+\frac{\beta }{2\left( \alpha +\beta \right) \left( b-x\right) }%
\int\limits_{x}^{b}\left[ t-\left( b-h\frac{b-a}{2}\right) \right] ^{2}dt.
\end{eqnarray*}%
After simple integration, we get%
\begin{eqnarray*}
&&\int\limits_{a}^{b}\left\vert P(x,t)\right\vert dt \\
&=&\frac{1}{6\left( \alpha +\beta \right) }\left[ 
\begin{array}{c}
\frac{\alpha }{x-a}\left\{ \left[ x-\left( a+h\frac{b-a}{2}\right) \right]
^{3}+\left( h\frac{b-a}{2}\right) ^{3}\right\} \\ 
-\frac{\beta }{b-x}\left\{ \left( h\frac{b-a}{2}\right) ^{3}-\left[ x-\left(
b-h\frac{b-a}{2}\right) \right] ^{3}\right\}%
\end{array}%
\right] .
\end{eqnarray*}%
Hence the first inequality is obtained.%
\begin{eqnarray*}
&&\left\vert \tau \left( x;\alpha ,\beta \right) \right\vert \\
&\leq &\left[ 
\begin{array}{c}
\frac{\alpha }{x-a}\left\{ \left[ x-\left( a+h\frac{b-a}{2}\right) \right]
^{3}+\left( h\frac{b-a}{2}\right) ^{3}\right\} \\ 
-\frac{\beta }{b-x}\left\{ \left( h\frac{b-a}{2}\right) ^{3}+\left[ x-\left(
b-h\frac{b-a}{2}\right) \right] ^{3}\right\}%
\end{array}%
\right] \frac{1}{6\left( \alpha +\beta \right) }\left\Vert f^{\text{ }\prime
\prime }\right\Vert _{\infty }.
\end{eqnarray*}%
Further, using H\"{o}lder's integral inequality in (\ref{13}) we have for $%
f^{\text{ }\prime \prime }\in L_{p}\left[ a,b\right] $%
\begin{equation*}
\left\vert \tau \left( x;\alpha ,\beta \right) \right\vert \leq \left\Vert
f^{\text{ }\prime \prime }\right\Vert _{p}\left(
\int\limits_{a}^{b}\left\vert P(x,t)\right\vert ^{q}dt\right) ^{\frac{1}{q}}.
\end{equation*}%
where $\frac{1}{p}+\frac{1}{q}=1$ and $p>1.$ Now%
\begin{eqnarray*}
&&\left( \alpha +\beta \right) \left( \int\limits_{a}^{b}\left\vert
P(x,t)\right\vert ^{q}dt\right) ^{\frac{1}{q}} \\
&=&\left[ 
\begin{array}{c}
\frac{\alpha ^{q}}{2^{q}\left( x-a\right) ^{q}}\int\limits_{a}^{x}\left[
t-\left( a+h\frac{b-a}{2}\right) \right] ^{2q}dt \\ 
+\frac{\beta ^{q}}{2^{q}\left( b-x\right) ^{q}}\int\limits_{x}^{b}\left[
t-\left( b-h\frac{b-a}{2}\right) \right] ^{2q}dt^{\frac{1}{q}}%
\end{array}%
\right] ^{\frac{1}{q}} \\
&=&\left[ 
\begin{array}{c}
\frac{\alpha ^{q}}{2^{q}\left( 2q+1\right) \left( x-a\right) ^{q}}\left[
t-\left( a+h\frac{b-a}{2}\right) \right] ^{2q+1}/_{a}^{x} \\ 
+\frac{\beta ^{q}}{2^{q}\left( 2q+1\right) \left( b-x\right) ^{q}}\left[
t-\left( b-h\frac{b-a}{2}\right) \right] ^{2q}/_{x}^{b}%
\end{array}%
\right] ^{\frac{1}{q}}.
\end{eqnarray*}%
Again, after simple integration, we get%
\begin{eqnarray*}
&&\left( \alpha +\beta \right) \left( \int\limits_{a}^{b}\left\vert
P(x,t)\right\vert ^{q}dt\right) ^{\frac{1}{q}} \\
&=&\frac{1}{2\left( 2q+1\right) ^{\frac{1}{q}}}\left[ 
\begin{array}{c}
\frac{\alpha ^{q}}{\left( x-a\right) ^{q}}\left\{ \left[ x-\left( a+h\frac{%
b-a}{2}\right) \right] ^{2q+1}-\left( h\frac{a-b}{2}\right) ^{2q+1}\right\}
\\ 
\\ 
+\frac{\beta ^{q}}{\left( b-x\right) ^{q}}\left\{ \left( h\frac{a-b}{2}%
\right) ^{2q+1}-\left[ x-\left( b-h\frac{b-a}{2}\right) \right]
^{2q+1}\right\}%
\end{array}%
\right] ^{\frac{1}{q}}.
\end{eqnarray*}%
Hence the second inequality is obtained as below.%
\begin{eqnarray*}
&&\left\vert \tau \left( x;\alpha ,\beta \right) \right\vert \\
&\leq &\frac{1}{2\left( 2q+1\right) ^{\frac{1}{q}}}\left[ 
\begin{array}{c}
\frac{\alpha ^{q}}{\left( x-a\right) ^{q}}\left\{ \left[ x-\left( a+h\frac{%
b-a}{2}\right) \right] ^{2q+1}-\left( h\frac{a-b}{2}\right) ^{2q+1}\right\}
\\ 
\\ 
+\frac{\beta ^{q}}{\left( b-x\right) ^{q}}\left\{ \left( h\frac{a-b}{2}%
\right) ^{2q+1}-\left[ x-\left( b-h\frac{b-a}{2}\right) \right]
^{2q+1}\right\}%
\end{array}%
\right] ^{\frac{1}{q}}\left\Vert f^{\text{ }\prime \prime }\right\Vert _{p}.
\end{eqnarray*}%
Finally, for $f^{\text{ }\prime \prime }\in L_{1}\left[ a,b\right] $, using (%
\ref{9}), we have the following inequality from (\ref{13}),%
\begin{equation*}
\left\vert \tau \left( x;\alpha ,\beta \right) \right\vert \leq \underset{%
t\in \left[ a,b\right] }{\sup }\left\vert P(x,t)\right\vert \left\Vert f^{%
\text{ }\prime \prime }\right\Vert _{1},
\end{equation*}%
where%
\begin{eqnarray*}
\left( \alpha +\beta \right) \underset{t\in \left[ a,b\right] }{\sup }%
\left\vert P(x,t)\right\vert &=&\frac{1}{4}\left\{ 
\begin{array}{c}
\alpha \left( x-a\right) +\beta \left( b-x\right) -h\left( b-a\right) \left(
\alpha +\beta \right) \\ 
\\ 
+\frac{h^{2}\left( b-a\right) }{2}\left[ \frac{\alpha }{x-a}+\frac{\beta }{%
b-x}\right]%
\end{array}%
\right\} \\
&&+%
\begin{array}{c}
\frac{1}{4}|\beta \left( b-x\right) -\alpha \left( x-a\right) +h\left(
b-a\right) \left( \alpha -\beta \right) \\ 
\\ 
+\frac{h^{2}\left( b-a\right) }{2}\left[ \frac{\beta }{b-x}-\frac{\alpha }{%
x-a}\right] |%
\end{array}%
\end{eqnarray*}%
This gives us the last inequality as below.%
\begin{equation*}
\left\vert \tau \left( x;\alpha ,\beta \right) \right\vert \leq \left\{ 
\begin{array}{c}
\alpha \left( x-a\right) +\beta \left( b-x\right) -h\left( b-a\right) \left(
\alpha +\beta \right) \\ 
+\frac{h^{2}\left( b-a\right) }{2}\left[ \frac{\alpha }{x-a}+\frac{\beta }{%
b-x}\right] \\ 
\begin{array}{c}
+|\beta \left( b-x\right) -\alpha \left( x-a\right) \\ 
+h\left( b-a\right) \left( \alpha -\beta \right) +\frac{h^{2}\left(
b-a\right) }{2}\left[ \frac{\beta }{b-x}-\frac{\alpha }{x-a}\right] |%
\end{array}%
\end{array}%
\right\} \frac{\left\Vert f^{\text{ }\prime \prime }\right\Vert _{1}}{%
4\left( \alpha +\beta \right) }.
\end{equation*}%
This completes the proof of theorem.
\end{proof}

\begin{quotation}
{\huge Some special cases of \ Theorem 2}
\end{quotation}

In this section, we will give some useful cases.

\begin{remark}
If we put $h=0$ in (\ref{12}), we get (\ref{6-B}).
\end{remark}

\begin{remark}
If we put $h=0$ in (\ref{6-A}), we get Cerone's result given in (\ref{6}).
\end{remark}

\begin{remark}
If we put $h=1$ in (\ref{12}), we get a new result.%
\begin{eqnarray}
&&\left\vert \tau \left( x;\alpha ,\beta \right) \right\vert  \label{14} \\
&&  \notag \\
&\leq &\left\{ 
\begin{array}{l}
\left[ 
\begin{array}{c}
\frac{\alpha }{x-a}\left\{ \left( x-A\right) ^{3}+\left( \frac{b-a}{2}%
\right) ^{3}\right\} \\ 
-\frac{\beta }{b-x}\left\{ \left( \frac{b-a}{2}\right) ^{3}+\left(
x-A\right) ^{3}\right\}%
\end{array}%
\right] \frac{1}{6\left( \alpha +\beta \right) }\left\Vert f^{\text{ }\prime
\prime }\right\Vert _{\infty },\text{ \ }f^{\text{ }\prime \prime }\in
L_{\infty }\left[ a,b\right] , \\ 
\\ 
\left[ 
\begin{array}{c}
\frac{1}{\left( x-a\right) ^{q}}\left\{ \left( x-A\right) ^{2q+1}-\left( 
\frac{a-b}{2}\right) ^{2q+1}\right\} \\ 
\\ 
+\frac{1}{\left( b-x\right) ^{q}}\left\{ \left( \frac{a-b}{2}\right)
^{2q+1}-\left( x-A\right) ^{2q+1}\right\}%
\end{array}%
\right] ^{\frac{1}{q}}\frac{1}{\left( 2q+1\right) ^{\frac{1}{q}}}\frac{1}{4}%
\left\Vert f^{\text{ }\prime \prime }\right\Vert _{p} \\ 
,f^{\text{ }\prime \prime }\in L_{p}\left[ a,b\right] ,p>1,\frac{1}{p}+\frac{%
1}{q}=1, \\ 
\\ 
\left\{ 
\begin{array}{c}
\alpha \left( x-a\right) +\beta \left( b-x\right) -\left( b-a\right) \left(
\alpha +\beta \right) \\ 
+\frac{b-a}{2}\left[ \frac{\alpha }{x-a}+\frac{\beta }{b-x}\right] \\ 
+%
\begin{array}{c}
|\beta \left( b-x\right) -\alpha \left( x-a\right) +\left( b-a\right) \left(
\alpha -\beta \right) \\ 
+\frac{b-a}{2}\left[ \frac{\beta }{b-x}-\frac{\alpha }{x-a}\right] |%
\end{array}%
\end{array}%
\right\} \frac{\left\Vert f^{\text{ }\prime \prime }\right\Vert _{1}}{%
4\left( \alpha +\beta \right) },f^{\text{ }\prime \prime }\in L_{1}\left[ a,b%
\right] ,%
\end{array}%
\right.  \notag
\end{eqnarray}%
where $A=\frac{a+b}{2}.$
\end{remark}

\begin{corollary}
If \ we put $x=A$ in above , we get%
\begin{eqnarray}
&&\left\vert \tau \left( A;\alpha ,\beta \right) \right\vert  \label{14-A} \\
&&  \notag \\
&\leq &\left\{ 
\begin{array}{l}
\frac{\left( b-a\right) ^{2}}{24}\frac{\alpha -\beta }{\alpha +\beta }%
\left\Vert f^{\text{ }\prime \prime }\right\Vert _{\infty },\text{ \ }f^{%
\text{ }\prime \prime }\in L_{\infty }\left[ a,b\right] , \\ 
\\ 
\left[ \frac{\beta ^{q}}{\left( b-a\right) ^{q}}\left( \frac{a-b}{2}\right)
^{2q+1}-\frac{\alpha ^{q}}{\left( b-a\right) ^{q}}\left( \frac{a-b}{2}%
\right) ^{2q+1}\right] ^{\frac{1}{q}}\frac{1}{\left( 2q+1\right) ^{\frac{1}{q%
}}}\frac{1}{\left( \alpha +\beta \right) }\left\Vert f^{\text{ }\prime
\prime }\right\Vert _{p} \\ 
,f^{\text{ }\prime \prime }\in L_{p}\left[ a,b\right] ,p>1,\frac{1}{p}+\frac{%
1}{q}=1, \\ 
\\ 
\left\{ \left( 1-\frac{a-b}{2}\right) +\left\vert \frac{\beta -\alpha }{%
\alpha +\beta }\left( 1-\frac{a-b}{2}\right) \right\vert \right\} \frac{%
\left\Vert f^{\text{ }\prime \prime }\right\Vert _{1}}{4},f^{\text{ }\prime
\prime }\in L_{1}\left[ a,b\right] .%
\end{array}%
\right.  \notag
\end{eqnarray}
\end{corollary}

\begin{remark}
If we put $h=\frac{1}{2}$ in (\ref{11}) and (\ref{12}) we get the following
inequality:%
\begin{eqnarray}
&&\left\vert 
\begin{array}{c}
\frac{1}{2\left( \alpha +\beta \right) }\left[ \frac{\alpha }{x-a}\left( x-%
\frac{3a+b}{4}\right) ^{2}-\frac{\beta }{b-x}\left( x-\frac{a+3b}{4}\right)
^{2}\right] f^{\text{ }\prime }\left( x\right) \\ 
-\frac{1}{\left( \alpha +\beta \right) }\left[ \frac{\alpha }{x-a}\left( x-%
\frac{3a+b}{4}\right) -\frac{\beta }{b-x}\left( x-\frac{a+3b}{4}\right) %
\right] f\left( x\right) \\ 
-\frac{1}{\alpha +\beta }\frac{b-a}{4}\left[ \frac{\alpha }{x-a}f\left(
a\right) +\frac{\beta }{b-x}f\left( b\right) \right] \\ 
+\frac{1}{\alpha +\beta }\frac{\left( b-a\right) ^{2}}{32}\left[ \frac{\beta 
}{b-x}f^{\text{ }\prime }\left( b\right) -\frac{\alpha }{x-a}f^{\text{ }%
\prime }\left( a\right) \right] \\ 
+\frac{1}{\alpha +\beta }\left[ \alpha M\left( f;a,x\right) +\beta M\left(
f;x,b\right) \right]%
\end{array}%
\right\vert  \label{15} \\
&&  \notag \\
&\leq &\left\{ 
\begin{array}{l}
\left[ 
\begin{array}{c}
\frac{\alpha }{x-a}\left\{ \left[ x-\left( a+\frac{b-a}{4}\right) \right]
^{3}+\left( \frac{b-a}{4}\right) ^{3}\right\} \\ 
-\frac{\beta }{b-x}\left\{ \left( \frac{b-a}{4}\right) ^{3}+\left[ x-\left(
b-\frac{b-a}{4}\right) \right] ^{3}\right\}%
\end{array}%
\right] \frac{\left\Vert f^{\text{ }\prime \prime }\right\Vert _{\infty }}{%
6\left( \alpha +\beta \right) },\text{ \ }f^{\text{ }\prime \prime }\in
L_{\infty }\left[ a,b\right] , \\ 
\\ 
\left[ 
\begin{array}{c}
\frac{\alpha ^{q}}{\left( x-a\right) ^{q}}\left\{ \left[ x-\left( a+\frac{b-a%
}{4}\right) \right] ^{2q+1}-\left( \frac{a-b}{4}\right) ^{2q+1}\right\} \\ 
+\frac{\beta ^{q}}{\left( b-x\right) ^{q}}\left\{ \left( \frac{a-b}{4}%
\right) ^{2q+1}-\left[ x-\left( b-\frac{b-a}{4}\right) \right]
^{2q+1}\right\}%
\end{array}%
\right] ^{\frac{1}{q}}\frac{\left\Vert f^{\text{ }\prime \prime }\right\Vert
_{p}}{2\left( 2q+1\right) ^{\frac{1}{q}}\left( \alpha +\beta \right) } \\ 
,f^{\text{ }\prime \prime }\in L_{p}\left[ a,b\right] ,p>1,\frac{1}{p}+\frac{%
1}{q}=1, \\ 
\\ 
\left\{ 
\begin{array}{c}
\alpha \left( x-a\right) +\beta \left( b-x\right) -\frac{1}{2}\left(
b-a\right) \left( \alpha +\beta \right) \\ 
+\frac{\left( b-a\right) }{8}\left[ \frac{\alpha }{x-a}+\frac{\beta }{b-x}%
\right] \\ 
+\left\vert 
\begin{array}{c}
\beta \left( b-x\right) -\alpha \left( x-a\right) +\frac{1}{2}\left(
b-a\right) \left( \alpha -\beta \right) \\ 
+\frac{\left( b-a\right) }{8}\left[ \frac{\beta }{b-x}-\frac{\alpha }{x-a}%
\right]%
\end{array}%
\right\vert%
\end{array}%
\right\} \frac{\left\Vert f^{\text{ }\prime \prime }\right\Vert _{1}}{%
4\left( \alpha +\beta \right) },f^{\text{ }\prime \prime }\in L_{1}\left[ a,b%
\right] .%
\end{array}%
\right.  \notag
\end{eqnarray}
\end{remark}

\begin{corollary}
If we put $\alpha =\beta $ and $x=A$ in (\ref{15}) we get another result.%
\begin{eqnarray}
&&%
\begin{array}{c}
|\frac{1}{2}f\left( \frac{a+b}{2}\right) +\frac{1}{4}\left[ f\left( a\right)
+f\left( b\right) \right] \\ 
-\frac{b-a}{32}\left[ f^{\text{ }\prime }\left( b\right) -f^{\text{ }\prime
}\left( a\right) \right] -(b-a)\int\limits_{a}^{b}f(t)dt|%
\end{array}
\label{16} \\
&&  \notag \\
&\leq &\left\{ 
\begin{array}{l}
\frac{\left( b-a\right) ^{2}\left\Vert f^{\text{ }\prime \prime }\right\Vert
_{\infty }}{192},\text{\ }f^{\text{ }\prime \prime }\in L_{\infty }\left[ a,b%
\right] , \\ 
\\ 
\left[ \left\{ \left( b-a\right) ^{2q+1}-\left( a-b\right) ^{2q+1}\right\} %
\right] ^{\frac{1}{q}}\frac{\left\Vert f^{\text{ }\prime \prime }\right\Vert
_{p}}{16\left( b-a\right) 2^{2.\frac{1}{q}}\left( 2q+1\right) ^{\frac{1}{q}}}
\\ 
,f^{\text{ }\prime \prime }\in L_{p}\left[ a,b\right] ,p>1,\frac{1}{p}+\frac{%
1}{q}=1, \\ 
\\ 
\frac{\left( b-a\right) \left\Vert f^{\text{ }\prime \prime }\right\Vert _{1}%
}{16},f^{\text{ }\prime \prime }\in L_{1}\left[ a,b\right] .%
\end{array}%
\right.  \notag
\end{eqnarray}
\end{corollary}

\begin{corollary}
If we put $\alpha =\beta $ and $x=\frac{a+3b}{4}$ in (\ref{15}) we get
another new result.%
\begin{eqnarray}
&&\left\vert 
\begin{array}{c}
\frac{b-a}{12}f^{\text{ }\prime }\left( x\right) -\left( \frac{1}{3}-\frac{2%
}{b-a}\right) f\left( x\right) -\frac{1}{2}\left( \frac{1}{3}f\left(
a\right) +f\left( b\right) \right) \\ 
+\frac{b-a}{16}\left( f^{\text{ }\prime }\left( b\right) -\frac{1}{3}f^{%
\text{ }\prime }\left( a\right) \right) \\ 
+\frac{2}{b-a}\left( \frac{1}{3}\int\limits_{a}^{\frac{a+3b}{4}%
}f(t)dt+\int\limits_{\frac{a+3b}{4}}^{b}f(t)dt\right)%
\end{array}%
\right\vert  \label{16-A} \\
&&  \notag \\
&\leq &\left\{ 
\begin{array}{l}
\frac{\left( b-a\right) ^{2}\left\Vert f^{\text{ }\prime \prime }\right\Vert
_{\infty }}{96},\text{ \ }f^{\text{ }\prime \prime }\in L_{\infty }\left[ a,b%
\right] , \\ 
\\ 
\left[ \frac{1}{3^{q}}\left\{ \left( \frac{b-a}{2}\right) ^{2q+1}-\left( 
\frac{a-b}{4}\right) ^{2q+1}\right\} +\left( \frac{a-b}{4}\right) ^{2q+1}%
\right] ^{\frac{1}{q}}\frac{\left\Vert f^{\text{ }\prime \prime }\right\Vert
_{p}}{\left( b-a\right) \left( 2q+1\right) ^{\frac{1}{q}}} \\ 
,f^{\text{ }\prime \prime }\in L_{p}\left[ a,b\right] ,p>1,\frac{1}{p}+\frac{%
1}{q}=1, \\ 
\\ 
\left\{ \frac{2}{3}+\left\vert \frac{b-a}{2}-\frac{1}{3}\right\vert \right\} 
\frac{\left\Vert f^{\text{ }\prime \prime }\right\Vert _{1}}{8},f^{\text{ }%
\prime \prime }\in L_{1}\left[ a,b\right] .%
\end{array}%
\right.  \notag
\end{eqnarray}
\end{corollary}

Hence, for different values of $h$, we can obtain a variety of results.

\begin{remark}
We can write (\ref{11}) in another way. Since%
\begin{eqnarray*}
&&\alpha M\left( f;a,x\right) +\beta M\left( f;x,b\right) \\
&=&\alpha M\left( f;a,x\right) +\frac{\beta }{b-x}\left( \overset{b}{%
\underset{a}{\int }}f\left( u\right) du-\overset{x}{\underset{a}{\int }}%
f\left( u\right) du\right) .
\end{eqnarray*}%
or%
\begin{eqnarray*}
&&\alpha M\left( f;a,x\right) +\beta M\left( f;x,b\right) \\
&=&\alpha M\left( f;a,x\right) -\frac{\beta }{b-x}\overset{x}{\underset{a}{%
\int }}f\left( u\right) du+\frac{\beta }{b-x}\overset{b}{\underset{a}{\int }}%
f\left( u\right) du \\
&=&\left( \alpha +\beta -\beta \sigma \left( x\right) \right) M\left(
f;a,x\right) +\beta \sigma \left( x\right) M\left( f;a,b\right) ,
\end{eqnarray*}%
where%
\begin{equation}
\frac{b-a}{b-x}=\sigma \left( x\right) .  \label{17}
\end{equation}%
Thus, from (\ref{11}),%
\begin{eqnarray}
&&\tau \left( x;\alpha ,\beta \right)  \label{18} \\
&=&%
\begin{array}{c}
\frac{1}{2\left( \alpha +\beta \right) }\left[ \frac{\alpha }{x-a}\left(
x-\left( a+h\frac{b-a}{2}\right) \right) ^{2}\right. \\ 
\left. -\frac{\beta }{b-x}\left( x-\left( b-h\frac{b-a}{2}\right) \right)
^{2}\right] f^{\text{ }\prime }\left( x\right)%
\end{array}
\notag \\
&&%
\begin{array}{c}
-\frac{1}{\left( \alpha +\beta \right) }\left[ \frac{\alpha }{x-a}\left(
x-\left( a+h\frac{b-a}{2}\right) \right) \right. \\ 
\left. -\frac{\beta }{b-x}\left( x-\left( b-h\frac{b-a}{2}\right) \right) %
\right] f\left( x\right)%
\end{array}
\notag \\
&&-\frac{1}{\alpha +\beta }h\frac{b-a}{2}\left[ \frac{\alpha }{x-a}f\left(
a\right) +\frac{\beta }{b-x}f\left( b\right) \right]  \notag \\
&&+\frac{1}{\alpha +\beta }h^{2}\frac{\left( b-a\right) ^{2}}{8}\left[ \frac{%
\beta }{b-x}f^{\text{ }\prime }\left( b\right) -\frac{\alpha }{x-a}f^{\text{ 
}\prime }\left( a\right) \right]  \notag \\
&&+\left[ \left( 1-\frac{\beta }{\alpha +\beta }\sigma \left( x\right)
\right) M\left( f;a,x\right) +\frac{\beta }{\alpha +\beta }\sigma \left(
x\right) M\left( f;a,b\right) \right] .  \notag
\end{eqnarray}%
so that for fixed $\left[ a,b\right] ,$ $M\left( f;a,b\right) $ is also
fixed.
\end{remark}

\begin{corollary}
If (\ref{11}) and (\ref{12}) is evaluated at $x=\frac{a+b}{2}$ and $\alpha
=\beta $ then%
\begin{eqnarray}
&&%
\begin{array}{c}
|\left( h-1\right) f\left( \frac{a+b}{2}\right) -\frac{h}{2}\left( f\left(
a\right) +f\left( b\right) \right) \\ 
+h^{2}\frac{b-a}{8}\left( f^{\text{ }\prime }\left( b\right) -f^{\text{ }%
\prime }\left( a\right) \right) +\frac{1}{b-a}\int\limits_{a}^{b}f\left(
t\right) dt|%
\end{array}
\label{19} \\
&&  \notag \\
&\leq &\left\{ 
\begin{array}{l}
\left( 1-h\right) ^{3}\frac{\left( b-a\right) ^{2}}{24}\left\Vert f^{\text{ }%
\prime \prime }\right\Vert _{\infty },\text{ \ }f^{\text{ }\prime \prime
}\in L_{\infty }\left[ a,b\right] , \\ 
\left[ 
\begin{array}{c}
\frac{2^{q}}{\left( b-a\right) ^{q}}\left\{ \left( \frac{b-a}{2}\left(
1-h\right) \right) ^{2q+1}-\left( h\frac{a-b}{2}\right) ^{2q+1}\right\} \\ 
\\ 
+\frac{2^{q}}{\left( b-a\right) ^{q}}\left\{ \left( h\frac{a-b}{2}\right)
^{2q+1}-\left( \frac{b-a}{2}\left( 1-h\right) \right) ^{2q+1}\right\}%
\end{array}%
\right] ^{\frac{1}{q}}\frac{\left\Vert f^{\text{ }\prime \prime }\right\Vert
_{p}}{4\left( 2q+1\right) ^{\frac{1}{q}}} \\ 
,f^{\text{ }\prime \prime }\in L_{p}\left[ a,b\right] ,p>1,\frac{1}{p}+\frac{%
1}{q}=1, \\ 
\\ 
\left[ \left( b-a\right) \left( 1-2h\right) +2h^{2}\right] \frac{\left\Vert
f^{\text{ }\prime \prime }\right\Vert _{1}}{8},f^{\text{ }\prime \prime }\in
L_{1}\left[ a,b\right] .%
\end{array}%
\right.  \notag
\end{eqnarray}
\end{corollary}

\section{\textbf{Perturbed Results}}

In 1882, \v{C}eby\^{s}ev \cite{4} gave the following inequality.%
\begin{equation}
\left\vert T\left( f,g\right) \right\vert \leq \frac{1}{12}\left( b-a\right)
^{2}\left\Vert f^{\prime }\right\Vert _{\infty }\left\Vert g^{\prime
}\right\Vert _{\infty }  \label{20}
\end{equation}%
where \ $f,g\ $: $\left[ a,b\right] \rightarrow 
\mathbb{R}
$ \ are\ absolutely continuous functions, which has bounded first
derivatives such that%
\begin{eqnarray}
T\left( f,g\right) &=&\frac{1}{b-a}\overset{b}{\underset{a}{\int }}f\left(
x\right) g\left( x\right) dx  \label{21} \\
&&-\left( \frac{1}{b-a}\overset{b}{\underset{a}{\int }}f\left( x\right)
dx\right) \left( \frac{1}{b-a}\overset{b}{\underset{a}{\int }}g\left(
x\right) dx\right)  \notag \\
&=&M\left( f,g;a,b\right) -M\left( f;a,b\right) M\left( g;a,b\right) , 
\notag
\end{eqnarray}%
and $\left\Vert .\right\Vert _{\infty }$ denotes the norm in $L_{\infty }%
\left[ a,b\right] $ defined as $\left\Vert p\right\Vert _{\infty }=ess%
\underset{t\in \left[ a,b\right] }{\sup }\left\vert P\left( t\right)
\right\vert .$

In 1935, Gr\"{u}ss \cite{9} proved the following inequality:%
\begin{eqnarray}
&&\left\vert \frac{1}{b-a}\overset{b}{\underset{a}{\int }}f\left( x\right)
g\left( x\right) dx-\frac{1}{b-a}\overset{b}{\underset{a}{\int }}f\left(
x\right) dx\frac{1}{b-a}\overset{b}{\underset{a}{\int }}g\left( x\right)
dx\right\vert  \label{22} \\
&\leq &\frac{1}{4}\left( \Phi -\varphi \right) \left( \Gamma -\gamma \right)
,  \notag
\end{eqnarray}%
provided that $f$ and $g$ are two integrable functions on $\left[ a,b\right] 
$ and satisfy the condition:%
\begin{equation}
\varphi \leq f\left( x\right) \leq \Phi \text{ \ and }\gamma \leq g\left(
x\right) \leq \Gamma ,\text{ for all }x\in \left[ a,b\right] .  \label{23}
\end{equation}%
The constant $\frac{1}{4}$ is best possible. The perturbed version of the
results of Theorem 2 can be obtained by using Gr\"{u}ss type results
involving the \v{C}eby\^{s}ev functional.%
\begin{equation*}
T\left( f,g\right) =M\left( f,g;a,b\right) -M\left( f;a,b\right) M\left(
g;a,b\right) ,
\end{equation*}%
where $M$ is the integral mean and is defined in (\ref{3}).

\begin{theorem}
Let $f\ $: $\left[ a,b\right] \rightarrow 
\mathbb{R}
$ \ be\ an absolutely continuous mapping and $\alpha ,$ $\beta $ are
non-negative real numbers, then%
\begin{eqnarray}
&&\left\vert \tau \left( x;\alpha ,\beta \right) -\frac{1}{\left( \alpha
+\beta \right) }\left[ 
\begin{array}{c}
\frac{\alpha }{x-a}\left\{ \left[ x-\left( a+h\frac{b-a}{2}\right) \right]
^{3}+\left( h\frac{b-a}{2}\right) ^{3}\right\} \\ 
-\frac{\beta }{b-x}\left\{ \left( h\frac{b-a}{2}\right) ^{3}+\left[ x-\left(
b-h\frac{b-a}{2}\right) \right] ^{3}\right\}%
\end{array}%
\right] \frac{\kappa }{6}\right\vert  \label{24} \\
&\leq &\left( b-a\right) N\left( x\right) \left[ \frac{1}{b-a}\left\Vert f^{%
\text{ }\prime \prime }\right\Vert _{2}^{2}-\kappa ^{2}\right] ^{\frac{1}{2}}
\notag \\
&\leq &\left( b-a\right) \left( \Gamma -\gamma \right) \lambda ,  \notag
\end{eqnarray}%
where, $\tau \left( x;\alpha ,\beta \right) $ is as given by (\ref{11}) and $%
\lambda =$ $\Phi -\varphi $. Let%
\begin{equation}
\kappa =\frac{f^{\text{ }\prime }(b)-f^{\text{ }\prime }(a)}{b-a}  \label{25}
\end{equation}%
then%
\begin{eqnarray}
N^{2}\left( x\right) &=&\frac{1}{20\left( \alpha +\beta \right) ^{2}}\left\{ 
\begin{array}{c}
\frac{\alpha ^{2}}{\left( x-a\right) ^{2}}\left[ \left( x-\left( a+h\frac{b-a%
}{2}\right) \right) ^{5}+\left( h\frac{b-a}{2}\right) ^{5}\right] \\ 
+\frac{\beta ^{2}}{\left( b-x\right) ^{2}}\left[ \left( h\frac{b-a}{2}%
\right) ^{5}-\left( x-\left( b-h\frac{b-a}{2}\right) \right) ^{5}\right]%
\end{array}%
\right\}  \label{26} \\
&&  \notag \\
&&-\left( \frac{1}{6\left( b-a\right) \left( \alpha +\beta \right) }\left[ 
\begin{array}{c}
\frac{\alpha }{x-a}\left\{ \left[ x-\left( a+h\frac{b-a}{2}\right) \right]
^{3}+\left( h\frac{b-a}{2}\right) ^{3}\right\} \\ 
-\frac{\beta }{b-x}\left\{ \left( h\frac{b-a}{2}\right) ^{3}+\left[ x-\left(
b-h\frac{b-a}{2}\right) \right] ^{3}\right\}%
\end{array}%
\right] \right) ^{2}  \notag
\end{eqnarray}
\end{theorem}

\begin{proof}
Associating $f\left( t\right) $ with $P(x,t)$ and $g\left( t\right) $ with $%
f^{\text{ }\prime \prime }(t)$, then from (\ref{9}) and (\ref{21}), we obtain%
\begin{equation*}
T\left( P\left( x,.\right) ,f^{\text{ }\prime \prime }(.);a,b\right)
=M\left( P\left( x,.\right) ,f^{\text{ }\prime \prime }(.);a,b\right)
-M\left( P\left( x,.\right) ;a,b\right) M\left( f^{\text{ }\prime \prime
}(.);a,b\right)
\end{equation*}%
Now using identity (\ref{10}),%
\begin{equation}
\left( b-a\right) T\left( P\left( x,.\right) ,f^{\text{ }\prime \prime
}(.);a,b\right) =\tau \left( x;\alpha ,\beta \right) -\left( b-a\right)
M\left( P\left( x,.\right) ;a,b\right) \kappa  \label{27}
\end{equation}%
where $\kappa $ is the secant slope of $f^{\text{ }\prime }$ over $\left[ a,b%
\right] $, as given in (\ref{25})$.$ Now, from (\ref{10}) and (\ref{21}),%
\begin{eqnarray}
&&\left( b-a\right) M\left( P\left( x,.\right) ;a,b\right)  \label{28} \\
&=&\int\limits_{a}^{b}P(x,t)dt  \notag \\
&=&\frac{\alpha }{2\left( \alpha +\beta \right) \left( x-a\right) }%
\int\limits_{a}^{x}\left[ t-\left( a+h\frac{b-a}{2}\right) \right] ^{2}dt 
\notag \\
&&+\frac{\beta }{2\left( \alpha +\beta \right) \left( b-x\right) }%
\int\limits_{x}^{b}\left[ t-\left( b-h\frac{b-a}{2}\right) \right] ^{2}dt 
\notag \\
&=&\frac{1}{6\left( \alpha +\beta \right) }\left[ 
\begin{array}{c}
\frac{\alpha }{x-a}\left\{ \left[ x-\left( a+h\frac{b-a}{2}\right) \right]
^{3}+\left( h\frac{b-a}{2}\right) ^{3}\right\} \\ 
-\frac{\beta }{b-x}\left\{ \left( h\frac{b-a}{2}\right) ^{3}+\left[ x-\left(
b-h\frac{b-a}{2}\right) \right] ^{3}\right\}%
\end{array}%
\right]  \notag
\end{eqnarray}%
Now combining (\ref{28}) with (\ref{26}) the left hand side of (\ref{24}) is
obtained$.$

Let $f,g\ $: $\left[ a,b\right] \rightarrow 
\mathbb{R}
$ and $fg\ $: $\left[ a,b\right] \rightarrow 
\mathbb{R}
$ be integrable on $\left[ a,b\right] ,$ then \cite{1}%
\begin{eqnarray}
\left\vert T\left( f,g\right) \right\vert &\leq &T^{\frac{1}{2}}\left(
f,f\right) T^{\frac{1}{2}}\left( g,g\right) \text{ \ \ \ \ \ \ }\left(
f,g\in L_{2}\left[ a,b\right] \right)  \label{29} \\
&\leq &\frac{\left( \Gamma -\gamma \right) }{2}T^{\frac{1}{2}}\left(
f,f\right) \text{ \ \ \ \ \ \ \ }\left( \gamma \leq g\left( x\right) \leq
\Gamma ,\text{ }t\in \left[ a,b\right] \right)  \notag \\
&\leq &\frac{1}{4}\left( \Phi -\varphi \right) \left( \Gamma -\gamma \right) 
\text{ \ \ \ \ }\left( \varphi \leq f\left( x\right) \leq \Phi ,\text{ }t\in %
\left[ a,b\right] \right) .  \notag
\end{eqnarray}

Also, note that%
\begin{eqnarray}
0 &\leq &T^{\frac{1}{2}}\left( f^{\text{ }\prime \prime }(.),f^{\text{ }%
\prime \prime }(.)\right)   \label{30} \\
&=&\left[ M\left( f^{\text{ }\prime \prime }(.)^{2};a,b\right) -M^{2}\left(
f^{\text{ }\prime \prime }(.);a,b\right) \right] ^{\frac{1}{2}}  \notag \\
&=&\left[ \frac{1}{b-a}\int\limits_{a}^{b}\left\Vert f^{\text{ }\prime
\prime }\left( t\right) \right\Vert ^{2}dt-\left( \frac{\int%
\limits_{a}^{b}f^{\text{ }\prime \prime }\left( t\right) dt}{b-a}\right) ^{2}%
\right] ^{\frac{1}{2}}  \notag \\
&=&\left[ \frac{1}{b-a}\left\Vert f^{\text{ }\prime \prime }\right\Vert
_{2}^{2}-\kappa ^{2}\right] ^{\frac{1}{2}}  \notag \\
&\leq &\frac{\left( \Gamma -\gamma \right) }{2}  \notag
\end{eqnarray}%
where $\gamma \leq f^{\text{ }\prime \prime }\left( t\right) \leq \Gamma
,t\in \left[ a,b\right] .$ Now, for the bounds on (\ref{27}), we have to
determine $T^{\frac{1}{2}}\left( P\left( x,.\right) ,P\left( x,.\right)
\right) $ and $\varphi \leq P\left( x,.\right) \leq \Phi $ from (\ref{29})
and (\ref{30})$.$

Now from (\ref{9}), the definition of $P(x,t),$ we have%
\begin{equation}
T\left( P\left( x,.\right) ,P\left( x,.\right) \right) =M\left( P^{2}\left(
x,.\right) ;a,b\right) -M^{2}\left( P\left( x,.\right) ;a,b\right) .
\label{31}
\end{equation}%
From (\ref{29}) we obtain%
\begin{eqnarray*}
&&M\left( P\left( x,.\right) ;a,b\right)  \\
&=&\frac{1}{6\left( \alpha +\beta \right) }\left[ 
\begin{array}{c}
\frac{\alpha }{x-a}\left\{ \left[ x-\left( a+h\frac{b-a}{2}\right) \right]
^{3}+\left( h\frac{b-a}{2}\right) ^{3}\right\}  \\ 
-\frac{\beta }{b-x}\left\{ \left( h\frac{b-a}{2}\right) ^{3}+\left[ x-\left(
b-h\frac{b-a}{2}\right) \right] ^{3}\right\} 
\end{array}%
\right] 
\end{eqnarray*}%
and%
\begin{eqnarray*}
&&\left( b-a\right) M\left( P^{2}\left( x,.\right) ;a,b\right)  \\
&=&\left( \frac{\alpha }{\alpha +\beta }\right) ^{2}\frac{1}{4\left(
x-a\right) ^{2}}\int\limits_{a}^{x}\left[ t-\left( a+h\frac{b-a}{2}\right) %
\right] ^{4}dt \\
&&+\left( \frac{\beta }{\alpha +\beta }\right) ^{2}\frac{1}{4\left(
b-x\right) ^{2}}\int\limits_{x}^{b}\left[ t-\left( b-h\frac{b-a}{2}\right) %
\right] ^{4}dt \\
&=&\left( \frac{\alpha }{\alpha +\beta }\right) ^{2}\frac{1}{20\left(
x-a\right) ^{2}}\left[ \left( x-\left( a+h\frac{b-a}{2}\right) \right)
^{5}+\left( h\frac{b-a}{2}\right) ^{5}\right]  \\
&&+\left( \frac{\beta }{\alpha +\beta }\right) ^{2}\frac{1}{20\left(
b-x\right) ^{2}}\left[ \left( h\frac{b-a}{2}\right) ^{5}-\left( x-\left( b-h%
\frac{b-a}{2}\right) \right) ^{5}\right]  \\
&=&\frac{1}{20\left( \alpha +\beta \right) ^{2}}\left\{ 
\begin{array}{c}
\frac{\alpha ^{2}}{\left( x-a\right) ^{2}}\left[ \left( x-\left( a+h\frac{b-a%
}{2}\right) \right) ^{5}+\left( h\frac{b-a}{2}\right) ^{5}\right]  \\ 
+\frac{\beta ^{2}}{\left( b-x\right) ^{2}}\left[ \left( h\frac{b-a}{2}%
\right) ^{5}-\left( x-\left( b-h\frac{b-a}{2}\right) \right) ^{5}\right] 
\end{array}%
\right\} 
\end{eqnarray*}%
Thus, substituting the above results into (\ref{31}) gives%
\begin{equation*}
0\leq N\left( x\right) =T^{\frac{1}{2}}\left( P\left( x,.\right) ,P\left(
x,.\right) \right) 
\end{equation*}%
which is given explicitly by (\ref{26}). Combining (\ref{27}), (\ref{31})
and (\ref{30}) give from the first inequality in (\ref{29}), the first
inequality in (\ref{24})$.$ Now utilizing the inequality in (\ref{30})
produces the second result in (\ref{24})$.$ Further, it may be noticed from
the definition of $P\left( x,t\right) $ in (\ref{9}) that for $\alpha ,$ $%
\beta \geq 0$, give%
\begin{eqnarray*}
\Phi  &=&\underset{t\in \left[ a,b\right] }{\sup }P\left( x,t\right)  \\
&=&\frac{1}{4\left( \alpha +\beta \right) }\left\{ 
\begin{array}{c}
\alpha \left( x-a\right) +\beta \left( b-x\right) -h\left( b-a\right) \left(
\alpha +\beta \right)  \\ 
+\frac{h^{2}\left( b-a\right) }{2}\left[ \frac{\alpha }{x-a}+\frac{\beta }{%
b-x}\right]  \\ 
+\left\vert 
\begin{array}{c}
\beta \left( b-x\right) -\alpha \left( x-a\right) +h\left( b-a\right) \left(
\alpha -\beta \right)  \\ 
+\frac{h^{2}\left( b-a\right) }{2}\left[ \frac{\beta }{b-x}-\frac{\alpha }{%
x-a}\right] 
\end{array}%
\right\vert 
\end{array}%
\right\} \text{ \ }
\end{eqnarray*}%
\begin{eqnarray*}
\varphi  &=&\underset{t\in \left[ a,b\right] }{\inf }P\left( x,t\right)  \\
&=&\frac{h^{2}\left( b-a\right) ^{2}}{8\left( \alpha +\beta \right) }\left\{ 
\frac{\alpha }{x-a}+\frac{\beta }{b-x}-\left\vert \frac{\alpha }{x-a}-\frac{%
\beta }{b-x}\right\vert \right\} 
\end{eqnarray*}%
where $\Phi -\varphi =\lambda .$
\end{proof}

\section{\textbf{An Application to the Cumulative Distribution Function}}

Let $X\in \left[ a,b\right] $ be a random variable with the cumulative
distributive function%
\begin{equation*}
F\left( x\right) =P_{r}\left( X\leq x\right) =\overset{x}{\underset{a}{\int }%
}f\left( u\right) du,
\end{equation*}%
where $f$ is the probability density function. In particular,%
\begin{equation*}
\overset{b}{\underset{a}{\int }}f\left( u\right) du=1.
\end{equation*}%
The following theorem holds.

\begin{theorem}
Let $X$ and $F$ be as above, then%
\begin{eqnarray}
&&%
\begin{array}{c}
|\frac{1}{2}\left[ \alpha \left( b-x\right) \left( x-\left( a+h\frac{b-a}{2}%
\right) \right) ^{2}-\beta \left( x-a\right) \left( x-\left( b-h\frac{b-a}{2}%
\right) \right) ^{2}\right] f^{\text{ }\prime }\left( x\right) \\ 
-\left[ \alpha \left( b-x\right) \left( x-\left( a+h\frac{b-a}{2}\right)
\right) -\beta \left( x-a\right) \left( x-\left( b-h\frac{b-a}{2}\right)
\right) \right] f\left( x\right) \\ 
-h\frac{b-a}{2}\left[ \alpha \left( b-x\right) f\left( a\right) +\beta
\left( x-a\right) f\left( b\right) \right] \\ 
+h^{2}\frac{\left( b-a\right) ^{2}}{8}\left[ \beta \left( x-a\right) f^{%
\text{ }\prime }\left( b\right) -\alpha \left( b-x\right) f^{\text{ }\prime
}\left( a\right) \right] \\ 
+\left[ \alpha \left( b-x\right) -\beta \left( x-a\right) \right] F\left(
x\right) +\beta \left( x-a\right) |%
\end{array}
\label{32} \\
&&  \notag \\
&\leq &\left\{ 
\begin{array}{l}
\frac{\left\Vert f^{\text{ }\prime \prime }\right\Vert _{\infty }}{6}\left[ 
\begin{array}{c}
\alpha \left( b-x\right) \left\{ \left[ x-\left( a+h\frac{b-a}{2}\right) %
\right] ^{3}+\left( h\frac{b-a}{2}\right) ^{3}\right\} \\ 
-\beta \left( x-a\right) \left\{ \left( h\frac{b-a}{2}\right) ^{3}+\left[
x-\left( b-h\frac{b-a}{2}\right) \right] ^{3}\right\}%
\end{array}%
\right] ,f^{\text{ }\prime \prime }\in L_{\infty }\left[ a,b\right] , \\ 
\\ 
\frac{\left( b-x\right) \left( x-a\right) \left\Vert f^{\text{ }\prime
\prime }\right\Vert _{p}}{2\left( 2q+1\right) ^{\frac{1}{q}}}\left[ 
\begin{array}{c}
\frac{\alpha ^{q}}{\left( x-a\right) ^{q}}\left\{ \left[ x-\left( a+h\frac{%
b-a}{2}\right) \right] ^{2q+1}-\left( h\frac{a-b}{2}\right) ^{2q+1}\right\}
\\ 
\\ 
+\frac{\beta ^{q}}{\left( b-x\right) ^{q}}\left\{ \left( h\frac{a-b}{2}%
\right) ^{2q+1}-\left[ x-\left( b-h\frac{b-a}{2}\right) \right]
^{2q+1}\right\}%
\end{array}%
\right] ^{\frac{1}{q}} \\ 
,,f^{\text{ }\prime \prime }\in L_{p}\left[ a,b\right] ,p>1,\frac{1}{p}+%
\frac{1}{q}=1, \\ 
\\ 
\frac{\left( b-x\right) \left( x-a\right) \left\Vert f^{\text{ }\prime
\prime }\right\Vert _{1}}{4}\left\{ 
\begin{array}{c}
\alpha \left( x-a\right) +\beta \left( b-x\right) -h\left( b-a\right) \left(
\alpha +\beta \right) \\ 
+\frac{h^{2}\left( b-a\right) }{2}\left[ \frac{\alpha }{x-a}+\frac{\beta }{%
b-x}\right] \\ 
+\left\vert 
\begin{array}{c}
\beta \left( b-x\right) -\alpha \left( x-a\right) +h\left( b-a\right) \left(
\alpha -\beta \right) \\ 
+\frac{h^{2}\left( b-a\right) }{2}\left[ \frac{\beta }{b-x}-\frac{\alpha }{%
x-a}\right]%
\end{array}%
\right\vert%
\end{array}%
\right\} \\ 
,f^{\text{ }\prime \prime }\in L_{1}\left[ a,b\right] .%
\end{array}%
\right. .  \notag
\end{eqnarray}
\end{theorem}

\begin{proof}
From $\left( 2.3\right) ,$ and by using the definition of Probability
Density Function, we have%
\begin{eqnarray*}
\tau \left( x;\alpha ,\beta \right) &:&=\frac{1}{2\left( \alpha +\beta
\right) }\left[ \frac{\alpha }{x-a}\left( x-\left( a+h\frac{b-a}{2}\right)
\right) ^{2}-\frac{\beta }{b-x}\left( x-\left( b-h\frac{b-a}{2}\right)
\right) ^{2}\right] f^{\text{ }\prime }\left( x\right) \\
&&-\frac{1}{\left( \alpha +\beta \right) }\left[ \frac{\alpha }{x-a}\left(
x-\left( a+h\frac{b-a}{2}\right) \right) -\frac{\beta }{b-x}\left( x-\left(
b-h\frac{b-a}{2}\right) \right) \right] f\left( x\right) \\
&&-\frac{1}{\alpha +\beta }h\frac{b-a}{2}\left[ \frac{\alpha }{x-a}f\left(
a\right) +\frac{\beta }{b-x}f\left( b\right) \right] \\
&&+\frac{1}{\alpha +\beta }h^{2}\frac{\left( b-a\right) ^{2}}{8}\left[ \frac{%
\beta }{b-x}f^{\text{ }\prime }\left( b\right) -\frac{\alpha }{x-a}f^{\text{ 
}\prime }\left( a\right) \right] \\
&&+\frac{1}{\alpha +\beta }\left[ \alpha M\left( f;a,x\right) +\beta M\left(
f;x,b\right) \right] , \\
&=&\frac{1}{2\left( \alpha +\beta \right) }\left[ \frac{\alpha }{x-a}\left(
x-\left( a+h\frac{b-a}{2}\right) \right) ^{2}-\frac{\beta }{b-x}\left(
x-\left( b-h\frac{b-a}{2}\right) \right) ^{2}\right] f^{\text{ }\prime
}\left( x\right) \\
&&-\frac{1}{\left( \alpha +\beta \right) }\left[ \frac{\alpha }{x-a}\left(
x-\left( a+h\frac{b-a}{2}\right) \right) -\frac{\beta }{b-x}\left( x-\left(
b-h\frac{b-a}{2}\right) \right) \right] f\left( x\right) \\
&&-\frac{1}{\alpha +\beta }h\frac{b-a}{2}\left[ \frac{\alpha }{x-a}f\left(
a\right) +\frac{\beta }{b-x}f\left( b\right) \right] \\
&&+\frac{1}{\alpha +\beta }h^{2}\frac{\left( b-a\right) ^{2}}{8}\left[ \frac{%
\beta }{b-x}f^{\text{ }\prime }\left( b\right) -\frac{\alpha }{x-a}f^{\text{ 
}\prime }\left( a\right) \right] \\
&&+\frac{1}{\alpha +\beta }\left\{ \left[ \frac{\alpha \left( b-x\right)
-\beta \left( x-a\right) }{\left( x-a\right) \left( b-x\right) }\right]
F\left( x\right) +\frac{\beta }{\left( b-x\right) }\right\}
\end{eqnarray*}

or%
\begin{eqnarray}
&&\left( \alpha +\beta \right) \left( x-a\right) \left( b-x\right) \tau
\left( x;\alpha ,\beta \right)  \label{33} \\
&=&\frac{1}{2}\left[ 
\begin{array}{c}
\alpha \left( b-x\right) \left( x-\left( a+h\frac{b-a}{2}\right) \right) ^{2}
\\ 
-\beta \left( x-a\right) \left( x-\left( b-h\frac{b-a}{2}\right) \right) ^{2}%
\end{array}%
\right] f^{\text{ }\prime }\left( x\right)  \notag \\
&&-\left[ 
\begin{array}{c}
\alpha \left( b-x\right) \left( x-\left( a+h\frac{b-a}{2}\right) \right) \\ 
-\beta \left( x-a\right) \left( x-\left( b-h\frac{b-a}{2}\right) \right)%
\end{array}%
\right] f\left( x\right)  \notag \\
&&-h\frac{b-a}{2}\left[ \alpha \left( b-x\right) f\left( a\right) +\beta
\left( x-a\right) f\left( b\right) \right]  \notag \\
&&+h^{2}\frac{\left( b-a\right) ^{2}}{8}\left[ \beta \left( x-a\right) f^{%
\text{ }\prime }\left( b\right) -\alpha \left( b-x\right) f^{\text{ }\prime
}\left( a\right) \right]  \notag \\
&&+\left[ \alpha \left( b-x\right) -\beta \left( x-a\right) \right] F\left(
x\right) +\beta \left( x-a\right)  \notag
\end{eqnarray}%
Now using (\ref{12}) and (\ref{33}), we get our required result (\ref{32}).
\end{proof}

Putting $\alpha =\beta =\frac{1}{2}$ in Theorem 5 gives the following result.

\begin{corollary}
Let $X$ be a random variable, $F\left( x\right) $ cumulative distributive
function and $f$ is a probability density function. Then%
\begin{eqnarray}
&&%
\begin{array}{c}
|\frac{1}{4}\left[ \left( b-x\right) \left( x-\left( a+h\frac{b-a}{2}\right)
\right) ^{2}-\left( x-a\right) \left( x-\left( b-h\frac{b-a}{2}\right)
\right) ^{2}\right] f^{\text{ }\prime }\left( x\right) \\ 
-\frac{1}{2}\left[ \left( b-x\right) \left( x-\left( a+h\frac{b-a}{2}\right)
\right) -\left( x-a\right) \left( x-\left( b-h\frac{b-a}{2}\right) \right) %
\right] f\left( x\right) \\ 
-h\frac{b-a}{4}\left[ \left( b-x\right) f\left( a\right) +\left( x-a\right)
f\left( b\right) \right] \\ 
+h^{2}\frac{\left( b-a\right) ^{2}}{16}\left[ \left( x-a\right) f^{\text{ }%
\prime }\left( b\right) -\left( b-x\right) f^{\text{ }\prime }\left(
a\right) \right] \\ 
+\frac{1}{2}\left[ \left( b-x\right) -\left( x-a\right) \right] F\left(
x\right) +\frac{1}{2}\left( x-a\right) | \\ 
\\ 
\end{array}
\label{34} \\
&\leq &\left\{ 
\begin{array}{l}
\left[ 
\begin{array}{c}
\left( b-x\right) \left\{ \left[ x-\left( a+h\frac{b-a}{2}\right) \right]
^{3}+\left( h\frac{b-a}{2}\right) ^{3}\right\} \\ 
-\left( x-a\right) \left\{ \left( h\frac{b-a}{2}\right) ^{3}+\left[ x-\left(
b-h\frac{b-a}{2}\right) \right] ^{3}\right\}%
\end{array}%
\right] \frac{\left\Vert f^{\text{ }\prime \prime }\right\Vert _{\infty }}{12%
},f^{\text{ }\prime \prime }\in L_{\infty }\left[ a,b\right] , \\ 
\\ 
\left[ 
\begin{array}{c}
\frac{1}{\left( x-a\right) ^{q}}\left\{ \left[ x-\left( a+h\frac{b-a}{2}%
\right) \right] ^{2q+1}-\left( h\frac{a-b}{2}\right) ^{2q+1}\right\} \\ 
\\ 
+\frac{1}{\left( b-x\right) ^{q}}\left\{ \left( h\frac{a-b}{2}\right)
^{2q+1}-\left[ x-\left( b-h\frac{b-a}{2}\right) \right] ^{2q+1}\right\}%
\end{array}%
\right] ^{\frac{1}{q}}\frac{\left( b-x\right) \left( x-a\right) }{4\left(
2q+1\right) ^{\frac{1}{q}}}\left\Vert f^{\text{ }\prime \prime }\right\Vert
_{p} \\ 
,f^{\text{ }\prime \prime }\in L_{p}\left[ a,b\right] ,p>1,\frac{1}{p}+\frac{%
1}{q}=1, \\ 
\\ 
\left\{ 
\begin{array}{c}
\frac{1}{2}\left( b-a\right) -h\left( b-a\right) +\frac{h^{2}\left(
b-a\right) }{4}\left[ \frac{1}{x-a}+\frac{1}{b-x}\right] \\ 
+\left\vert \frac{1}{2}\left( a+b-2x\right) +\frac{h^{2}\left( b-a\right) }{4%
}\left[ \frac{1}{b-x}-\frac{1}{x-a}\right] \right\vert%
\end{array}%
\right\} \frac{\left( b-x\right) \left( x-a\right) \left\Vert f^{\text{ }%
\prime \prime }\right\Vert _{1}}{4} \\ 
,f^{\text{ }\prime \prime }\in L_{1}\left[ a,b\right] .%
\end{array}%
\right.  \notag
\end{eqnarray}
\end{corollary}

\begin{remark}
The above result allow the approximation of $F\left( x\right) $ in terms of $%
f\left( x\right) $. The approximation of%
\begin{equation*}
R\left( x\right) =1-F\left( x\right)
\end{equation*}%
could also be obtained by a simple substitution. $R\left( x\right) $ is of
importance in reliability theory where $f\left( x\right) $ is the
probability density function of failure.
\end{remark}

\begin{remark}
We put $\beta =0$ in (\ref{32}), assuming that $\alpha \neq 0$ to obtain%
\begin{eqnarray}
&&\left\vert 
\begin{array}{c}
\alpha \left( b-x\right) \left\{ 
\begin{array}{c}
\frac{1}{2}\left[ \left( x-\left( a+h\frac{b-a}{2}\right) \right) ^{2}\right]
f^{\text{ }\prime }\left( x\right) \\ 
-\left[ \left( x-\left( a+h\frac{b-a}{2}\right) \right) \right] f\left(
x\right) \\ 
-h\frac{b-a}{2}f\left( a\right) -h^{2}\frac{\left( b-a\right) ^{2}}{8}\left[
f^{\text{ }\prime }\left( a\right) \right] +F\left( x\right)%
\end{array}%
\right\} \\ 
\end{array}%
\right\vert  \label{35} \\
&\leq &\left\{ 
\begin{array}{l}
\left[ \alpha \left( b-x\right) \left\{ \left[ x-\left( a+h\frac{b-a}{2}%
\right) \right] ^{3}+\left( h\frac{b-a}{2}\right) ^{3}\right\} \right] \frac{%
\left\Vert f^{\text{ }\prime \prime }\right\Vert _{\infty }}{6},f^{\text{ }%
\prime \prime }\in L_{\infty }\left[ a,b\right] , \\ 
\\ 
\left[ \frac{\alpha ^{q}}{\left( x-a\right) ^{q}}\left\{ \left[ x-\left( a+h%
\frac{b-a}{2}\right) \right] ^{2q+1}-\left( h\frac{a-b}{2}\right)
^{2q+1}\right\} \right] ^{\frac{1}{q}}\frac{\left( b-x\right) \left(
x-a\right) \left\Vert f^{\text{ }\prime \prime }\right\Vert _{p}}{2\left(
2q+1\right) ^{\frac{1}{q}}} \\ 
,f^{\text{ }\prime \prime }\in L_{p}\left[ a,b\right] ,p>1,\frac{1}{p}+\frac{%
1}{q}=1, \\ 
\\ 
\left( 
\begin{array}{c}
\alpha \left( x-a\right) -h\alpha \left( b-a\right) +\frac{h^{2}\left(
b-a\right) }{2}\left( \frac{\alpha }{x-a}\right) \\ 
+\left\vert -\alpha \left( x-a\right) +h\alpha \left( b-a\right) -\frac{%
h^{2}\left( b-a\right) }{2}\left( \frac{\alpha }{x-a}\right) \right\vert%
\end{array}%
\right) \frac{\left( b-x\right) \left( x-a\right) \left\Vert f^{\text{ }%
\prime \prime }\right\Vert _{1}}{4} \\ 
,f^{\text{ }\prime \prime }\in L_{1}\left[ a,b\right] .%
\end{array}%
\right.  \notag
\end{eqnarray}%
\bigskip
\end{remark}

We may replace $f$ \ by $F$ in any of the equations (\ref{32})$,$(\ref{34}%
)and (\ref{35})\ so that the bounds are in terms of $\left\Vert f^{\text{ }%
\prime \prime }\right\Vert _{p},$ $p\geq 1.$ Further we note that%
\begin{equation*}
\overset{b}{\underset{a}{\int }}F\left( u\right) du=\left. uF\left( u\right)
\right\vert _{a}^{b}-\overset{b}{\underset{a}{\int }}xf\left( x\right)
dx=b-E\left( X\right) .
\end{equation*}

\textbf{Competing interests:}

The authors declare that they have no competing interests.

\textbf{Authors' contributions:}

All authors have contributed equally and significantly in writing this
article.

\end{document}